\documentclass{article}

\usepackage{amscd,amsmath,amssymb,amsthm}
\usepackage{color,cancel}
\usepackage{graphicx} 

\usepackage{mathrsfs}

\newcommand{\R}{\mathbb R}
\newcommand{\N}{\mathbb N}

\newcommand{\J}{{\cal J}}

\DeclareMathOperator*{\esssup}{ess\; sup}

\newtheorem{corollary}{Corollary}[section]
\newtheorem{theorem}[corollary]{Theorem}
\newtheorem{lemma}[corollary]{Lemma}
\newtheorem{proposition}[corollary]{Proposition}

\theoremstyle{definition}
\newtheorem{definition}[corollary]{Definition}
\newtheorem{remark}[corollary]{Remark}

\numberwithin{equation}{section}
\begin{document}

\title{{\bf{Multiple solutions for quasilinear elliptic problems with concave and convex nonlinearities}}
\footnote{The authors are members of the Research Group INdAM-GNAMPA. The second author is partially supported by the project PRIN PNRR 2022 "Linear and nonlinear PDE's: New directions and applications", CUP: H53D23008950001, and by the project GNAMPA 2024 "Nonlinear problems in local and nonlocal setting with applications",  CUP: E53C23001670001.}}

\author{Federica Mennuni$^a$ and Addolorata Salvatore$^b$\\ \\
{\small $^a$Dipartimento di Matematica} \\
{\small Universit\`a di Bologna} \\
{\small Via Zamboni, 33, 40126 Bologna, Italy}\\
{\small $^b$Dipartimento di Matematica} \\
{\small Universit\`a degli Studi di Bari Aldo Moro} \\
{\small Via E. Orabona 4, 70125 Bari, Italy}\\
{\small federica.mennuni@unibo.it}\\
{\small addolorata.salvatore@uniba.it}}

\date{}

\maketitle

\begin{abstract}
We prove the existence of multiple signed bounded solutions for a quasilinear elliptic equation with concave and convex nonlinearities. For this, we use a variational approach in an intersection Banach space indroduced by Candela and Palmieri, and a truncation technique given by   Garcia Azorero and Peral.
\end{abstract}



\section{Introduction}\label{section1}

We look for solutions to the generalized quasilinear problem 
\begin{equation}\label{problema1}
\left\{
\begin{array}{ll}
- {\rm div} (a(x,u,\nabla u)) + A_t (x, u,\nabla u)= g(x,u) &\hbox{ in $\Omega$,}\\ [10pt]
\quad u = 0 &\hbox{ on $\partial\Omega$,}
\end{array}
\right.
\end{equation}
where $\Omega$ is an open bounded domain, $A :\Omega \times \R \times {\R}^{N}\to \R$ is ${{C}}^{1}$-Carath\'eodory function with partial derivatives
\[
A_t(x,t,\xi) =\ \frac{\partial A}{\partial t}(x,t,\xi), \quad a(x,t,\xi) =\nabla_\xi A(x,t,\xi)
\]
and $g:\Omega \times \R \to \R$ is a given Carathéodory function.
Problem \eqref{problema1} has a variational structure since its solutions coincide, at least formally, with the critical points of the functional
\[
\int_{\Omega} A(x, u, \nabla u)dx -\int_{\Omega} G(x, u) dx
\]
(with $G(x,t)=\displaystyle\int_{0}^{t} g(x,s) ds$), whose formal Eulero-Lagrange equation is precisely \eqref{problema1}. Such a functional presents a loss of regularity, which arises independently on assumptions for $G$; indeed, if $G\equiv 0$, even in the simple case $A(x,t,\xi)=\frac{1}{p}\mathcal{A}(x,t)|\xi|^p$, $p>1$ with $\mathcal{A}(x,t)$ smooth, bounded and far away from zero, it is G\^ateaux differentiable only along directions of the Banach space $W_{0}^{1,p}(\Omega)\cap L^{\infty}(\Omega)$.

For this reason, different abstract approaches have been introduced in \cite{AB, CP1, CP2, CP3, Ca, JJ}. In particular, Candela and Palmieri introduced a weak Cerami-Palais-Smale condition (see Definition \ref{wCPSdef} below) that allows them to state a version of the Deformation Lemma and then a generalized version of the Minimum Princile and of the Mountain Pass Theorem.

Thus, under suitable assumptions for $A(x,t,\xi)$, if $g(x,\cdot)$ has a super$-p-$linear growth the existence of a bounded solution of \eqref{problema1} of mountain pass type has been stated in \cite{CP2}, while if $g(x,\cdot)$ has a sub$-p-$linear growth a bounded solution has been obtained in \cite{CS1} as a minimum of the action functional. Analogous results hold even if the problem is settled in $\R^N$, as proved in \cite{CS2, CSS, MS1, MS2}.

In this paper we want to study the quasilinear elliptic problem \eqref{problema1} when $g(x,\cdot)$ is the sum of a parametric sub$-p-$linear and a super$-p-$linear term. For simplicity, in this paper we consider the sum of two power terms, but more general function $g$ can be handled with the same behaviour at $0$ and at infinity. More precisely, we look for bounded solutions of the problem
\begin{equation}\label{problema2}
\left\{
\begin{array}{ll}
- {\rm div} (a(x,u,\nabla u)) + A_t (x, u,\nabla u)= \lambda |u|^{q-2}u+|u|^{s-2}u &\hbox{ in $\Omega$,}\\ [10pt]
\quad u = 0 &\hbox{ on $\partial\Omega$,}
\end{array}
\right.
\end{equation}
where $\lambda>0$ and $1<q<p<s<p^{*}$, where $p^{*}=\frac{Np}{N-p}$ if $N>p$, $p^{*}=+\infty$ if $N\leq p$. If $A(x,t,\xi)=\frac{1}{p}|\xi|^p, p>1$, the problem has been already studied by many authors (see, e.g.,  \cite{ABC, deF1, deF2, Per} if $p=2$ and \cite{BEP, deF1, deF3, GAP1} if $p\not = 2$).

By combining variational tools introduced in \cite{CP1, CP2} and truncation arguments in \cite{GAP}, we will state a multiplicity result for the problem \eqref{problema2} when $\lambda>0$ is sufficiently small.

The paper is organized as follows: in Section \ref{section2} we introduce the abstract framework and the general assumptions, then we state our main results and we give the variational approach. In Section \ref{section3} we present a truncation method and we prove the existence, for $\lambda$ small enough, of a negative local minimum of the action functional, while in Section \ref{section4} we find a mountain pass critical point. Finally, standard arguments allow us to state the existence of signed bounded solutions.


\section{Abstract setting and main result}
\label{section2}
In this section we assume that
\begin{itemize}
\item $(X, \|\cdot\|_X)$ is a Banach space with dual 
$(X',\|\cdot\|_{X'})$;
\item $(S,\|\cdot\|_S)$ is a Banach space such that
$X \hookrightarrow S$ continuously, i.e., $X \subset S$ and a constant $\sigma_0 > 0$ exists
such that
\[
\|u\|_S \ \le \ \sigma_0\ \|u\|_X\qquad \hbox{for all $u \in X$;}
\]
\item $J : {\cal D} \subset S \to \R$ and $J \in C^1(X,\R)$ with $X \subset {\cal D}$.
\end{itemize}
Just to fix the ideas, one may think to $X$ as the intersection of a Sobolev space with $L^\infty$ and to $S$ as the Sobolev space itself.

If $\beta \in \R$, we say that a sequence
$(u_n)_n\subset X$ is a {\sl Cerami--Palais--Smale sequence at level $\beta$},
briefly {\sl $(CPS)_\beta$--sequence}, if
\begin{equation*}
\lim_{n \to +\infty}J(u_n) = \beta\quad\mbox{and}\quad 
\lim_{n \to +\infty}\|dJ\left(u_n\right)\|_{X'} (1 + \|u_n\|_X) = 0.
\end{equation*}

Moreover, $\beta$ is a {Cerami--Palais--Smale level}, briefly {{\sl $(CPS)$--level}, if there exists a $(CPS)_\beta$--sequence. We say that $J$ satisfies the Cerami--Palais--Smale condition in $X$ at level $\beta$ if every $(CPS)_\beta$--sequence converges in $X$ (up to subsequences). However, thinking about the setting of our problem, it may happen that a $(CPS)_\beta$--sequence is unbounded in $\|\cdot\|_X$ but being convergent with respect to $\|\cdot\|_S$. Thus, according to the ideas developed in previous papers (see, for example, \cite{CP1, CP3}) we need to weaken the classical Cerami--Palais--Smale condition as follows.

\begin{definition} \label{wCPSdef}
The functional $J$ satisfies the
{\slshape weak Cerami--Palais--Smale 
condition at level $\beta \in \R$}, 
briefly {\sl $(wCPS)_\beta$ condition}, if for every $(CPS)_\beta$--sequence $(u_n)_n$,
a point $u \in X$ exists such that 
\begin{description}{}{}
\item[{\sl (i)}] $\displaystyle 
\lim_{n \to+\infty} \|u_n - u\|_S = 0\quad$ (up to subsequences),
\item[{\sl (ii)}] $J(u) = \beta$, $\; dJ(u) = 0$.
\end{description}
If $J$ satisfies the $(wCPS)_\beta$ condition at each level $\beta \in I$, $I$ real interval, 
we say that $J$ satisfies the $(wCPS)$ condition in $I$.
\end{definition}

Let us point out that, due to the convergence only in $S$, the {\sl $(wCPS)_\beta$ }condition  implies that the set of critical points of $J$ at the $\beta$ level is compact with respect to $\|\cdot\|_S$; but this is enough to prove a Deformation Lemma and some abstract theorems about the existence of critical points. In particular, the following generalized versions of the Minimum Principle and of the Mountain Pass Theorem apply.

\begin{proposition}
\label{Minimum Principle}
If $J \in {C}^{1}(X,\R)$ is bounded from below in $X$ and {\sl $(wCPS)_\beta$} holds at level $\beta =\displaystyle{\inf_{X} J \in \R}$, then $J$ attains its infimum, i.e.,  ${u}_0 \in X$ exists such that $J({u}_0)= \beta$ and $dJ(u_0)=0$.
\end{proposition}
\begin{proof}
See, \cite[Theorem 1.6]{CP3}.
\end{proof}

\begin{proposition}
\label{Mountain Pass Theorem}
Let $J \in {C}^{1}(X,\R)$ and suppose that $\rho>0, \alpha \in \R, u_0$ and $e\in X$ exist such that
\begin{itemize}
\item[(i)]$ u \in X, \|u-u_0\|_S=\rho \Longrightarrow J(u)\geq \alpha > J(u_0);$
\item[(ii)] $\|e-u_0\|_S>\rho$ and $J(e)<J(u_0).$
\end{itemize} 

Setting 
\[
c=\inf_{\varphi\in\Gamma}\max_{s\in \left[0,1\right]}J(\varphi(s))
\]
with $\Gamma=\{\varphi\in \mathcal{C}(\left[0,1\right], X): \varphi(0)=u_0, \varphi(1)=e\}$, if $J$ satisfies {\sl $(wCPS)_c$}, then $J$ has a critical point $u^*$ (different from $u_0$) such that $J(u^*)=c\geq \alpha>J(u_0).$ 
\end{proposition}
\begin{proof}
See, \cite[Theorem 1.7]{CP3}.
\end{proof}
\smallskip

From now on, let $\Omega$ be an open bounded domain of $\R^N$, $N\geq2$, and consider a function $A:\Omega\times\R\times \R^N\to \R$ with partial derivatives $A_t(x,t,\xi)$ and $a(x,t,\xi)$ according to the notations in Section \ref{section1}.

Assume that a real number $p>1$ exists such that the following assumptions hold: 
\begin{itemize}
\item[$(H_0)$]
$A$  is a ${C}^1$--Carath\'eodory function, i.e., $A(\cdot, t, \xi)$ is measurable for all $(t,\xi) \in \R \times \R^N$ and $A(x,\cdot, \cdot)$ is ${C}^1$ for a.e. $x \in \Omega$ ;
\item[$(H_1)$] there exist some positive continuous functions $ {{\Phi}}_{i}, \phi_i: \R \to \R,  i \in \{1,2\}$ such that
\[
\begin{split}
\vert A_t(x,t,\xi) \vert \quad \leq & \quad {\Phi}_{1}(t) + {\phi}_{1}(t) {\vert t \vert}^{p} \quad \quad \mbox{ a.e. $x\in \Omega$ and for all $(t,\xi) \in \R \times {\R}^{N},$} \\
\vert a(x,t,\xi) \vert \quad \leq & \quad {\Phi}_{2}(t) + {\phi}_{2}(t) {\vert t \vert}^{p-1}  \quad \mbox{ a.e. $x\in \Omega$ and for all $(t,\xi) \in \R \times {\R}^{N}$};
\end{split}
\]
\item[$(H_2)$] 
there exist a constant $\alpha_1 > 0$  such that
\[
A(x,t,\xi) \ge \alpha_1|\xi|^p
\quad \hbox{ a.e. $x\in\Omega$ and for all $(t,\xi) \in \R\times\R^N$;}
\]

\item[$(H_3)$]
there exists $\eta_1 > 0$ such that
\[
A(x,t,\xi) \leq \eta_1 \ a(x,t,\xi) \cdot \xi
\quad \hbox{ a.e. $x\in\Omega$ and for all $(t,\xi) \in \R\times\R^N$;}
\]
\item[$(H_4)$] 
there exists a constant $\alpha_2>0$ such that
\[
a(x,t,\xi)\cdot \xi + A_t(x,t,\xi) t \ge \alpha_2 a(x,t,\xi)\cdot \xi
\quad \hbox{ a.e. $x\in\Omega$ and for all $(t,\xi) \in \R\times\R^N$;}
\]
\item[$(H_5)$] there exist some constants $\eta_2, \delta>0$ such that 
\[
 A(x,t,\xi) \leq \eta_2 |\xi|^p
\quad \hbox{ a.e. $x\in\Omega$ and for all $(t,\xi) \in \R\times\R^N, |t|\leq \delta$;}
\]
\item[$(H_6)$]
for all $\xi$, $\xi^* \in \R^N$, $\xi \ne \xi^*$, it is
\[
[a(x,t,\xi) - a(x,t,\xi^*)]\cdot [\xi - \xi^*] > 0 \quad \hbox{ a.e. $x\in\Omega$ and for all $t\in \R$.}
\]
\end{itemize}
\begin{remark}\label{remark1}
From $(H_1)$--$(H_3)$ it follows that there exist two positive continuous functions $\Phi_0, \phi_0:\R\to\R$ such that
\[
|A(x,t,\xi)|=A(x,t,\xi)\leq \Phi_0(t)+\phi_0(t)|\xi|^p \quad \mbox{ a.e. in $\Omega$, for all $(t,\xi)\in \R\times\R^N$.}
\]
\end{remark}
\begin{remark}\label{remark2}
From $(H_2)$ and $(H_3)$ it follows that 
\[
a(x,t,\xi)\cdot \xi\geq \frac{\alpha_1}{\eta_1}|\xi|^p \quad \mbox{ a.e. in $\Omega$, for all $(t,\xi)\in \R\times\R^N$.}
\]
\end{remark}
We are ready to state our main results.
\begin{theorem}\label{theorem1}
Assume that $(H_0)$--$(H_6)$ hold. Then, there exists $\lambda_1>0$ such that for any $\lambda\in \left]0,\lambda_1\right]$ problem \eqref{problema2} has at least two weak bounded solutions, one positive and one negative.
\end{theorem}
\begin{theorem}\label{theorem2}
Under the same assumptions of Theorem \ref{theorem1}, if moreover
\begin{itemize}
\item[$(H_7)$] there exists a constant $\alpha_3>0$ such that
\[
sA(x,t,\xi)-a(x,t,\xi)\cdot \xi-A_t(x,t,\xi)t\geq \alpha_3a(x,t,\xi)\cdot\xi \quad \mbox{ a.e. in $\Omega$, for all $(t,\xi)\in \R\times\R^N$,}
\]
\end{itemize}
then, for any $\lambda\in \left]0,\lambda_1\right]$ problem \eqref{problema2} has at least four weak bounded solutions, two positive and two negative.
\end{theorem}
In order to give the variational formulation of the problem we denote by:
\begin{itemize}
\item $|\cdot|$ the Euclidean norm in $\R^N$ and  $x \cdot y$ the inner product of $x,y\in \R^N$;
\item $L^\ell(\Omega)$ the Lebesgue space with
norm $|u|_\ell = \displaystyle{ \left(\int_{\Omega}|u|^\ell dx\right)^{1/\ell}}$ if $1 \le \ell < +\infty$ and $|u|_{\infty} = \displaystyle\esssup_{\Omega} |u|$ if $\ell=\infty$;
\item $W^{1,p}(\Omega)$ the classical Sobolev space with norm $\|u\|_{p} = |\nabla u|_p$.
\end{itemize}
We set
\[
X=W_{0}^{1,p}(\Omega)\cap L^{\infty}(\Omega)
\]
equipped with the norm $\|u\|_X=\|u\|_p+|u|_{\infty}$.

From the Sobolev Embedding Theorems (see, e.g., \cite{Br}), for any $\ell \in [1,p^*]$ a constant $\sigma_{\ell}>0$ exists such that
\begin{equation}\label{condizione1}
|u|_{\ell}\leq \sigma_{\ell}\|u\|_p \quad \mbox{ for all $u\in W_{0}^{1,p}(\Omega)$}
\end{equation}
and the embedding $W_{0}^{1,p}(\Omega)\hookrightarrow \hookrightarrow L^{\ell}(\Omega)$ is compact if $p\leq\ell<p^{*}$. On the other hand, if $\ell>N$, then $W_{0}^{1,p}(\Omega)$ is continuous embedded in $L^{\infty}(\Omega)$.

By definition, $X \hookrightarrow W_{0}^{1,p}(\Omega)$ and  $X \hookrightarrow L^{\infty}(\Omega)$ with continuous embeddings.
From now on , we assume $1<p\leq N$ (otherwise,  it is $X=W_{0}^{1,p}(\Omega)$ and the following arguments can be simplified).

Now, we consider the functional
\[
\J(u)=\int_{\Omega} A(x,u,\nabla u) dx -\frac{\lambda}{q}\int_{\Omega}|u|^q \ dx -\frac{1}{s}\int_{\Omega} |u|^{s} dx, \quad u\in X.
\] 
Taking $u,v \in X$, by direct computations it follows that its G\^ateaux differential in $u$ along the direction $v$ is
\begin{equation}\label{differenziale1}
\left\langle d\J(u),v\right\rangle=\int_{\Omega} a(x,u,\nabla u)\cdot\nabla v dx+\int_{\Omega} A_t(x,u,\nabla u)v dx-\lambda \int_{\Omega}|u|^{q-2}uv \ dx-\int_{\Omega}|u|^{s-2}uv \ dx.
\end{equation} 
The following proposition states the regularity of $\J$ in $X$.
\begin{proposition}\label{proposition1}
Let us assume that conditions $(H_0)$, $(H_1)$ hold. Moreover, suppose that two positive continuous functions $\Phi_0$ and $\phi_0$ exists such that
\[
A(x,t,\xi)\leq \Phi_0(t)+\phi_0(t)|\xi|^p \quad \mbox{ a.e. in $\Omega$, for all $(t,\xi)\in \R\times\R^N$.}
\]
Then, if $(u_n)_n \subset X$ and $u\in X$ are such that
\[
\|u_n - u\|_p \to 0, \quad u_n \to u \ \quad\hbox{a.e. in $\Omega$, if $n \to+\infty$,}
\]
and a constant $M>0$ exists so that
\[
|u_n|_\infty \le M\quad \hbox{for all $n \in \N$,}
\]
it follows that $\J(u_n) \to \J(u)$ and $\|d\J(u_n)-d\J(u)\|_{X^{'}}\to 0 \quad \mbox{ if $n \to +\infty$.}$ Hence, $\J$ is a $\mathcal{C}^{1}–$ functional on $X$ with Fr\'echet differential as in \eqref{differenziale1}. 
\end{proposition}
\begin{proof}
We note that, since $1<q<p<s$ it is
\begin{equation}\label{condizione2}
|\lambda|t|^{q-2}t+|t|^{s-2}t|\leq (\lambda+1)\left(1+|t|^{s-1}\right) \quad \mbox{ for all $t\in \R$.}
\end{equation}
Thus, the proof follows arguing as in \cite[Proposition 3.1]{CP2}.
\end{proof}
\begin{remark}\label{remark3}
The previous proposition implies that the research of weak bounded solutions of \eqref{problema2} is reduced to the study of critical points of $\J$ in $X$.
\end{remark}
\section{Existence of a local minimum of $\J$} \label{section3}

Let us point out that $\J$ is not bounded from below in $X$, due to the presence of the term $\int_{\Omega}|u|^{s} dx$, $s>p$. Anyway, we can introduce a ''good'' bounded from below truncation of $\J$.

By $(H_2)$ and \eqref{condizione1} we have that
\[
\J(u)\geq \alpha_1 \int_{\Omega}|\nabla u|^p dx -\frac{\lambda}{q}c_q^{q}\left(\int_{\Omega}|\nabla u|^p dx\right)^{\frac{q}{p}}-\frac{1}{s}c_s^{s}\left(\int_{\Omega}|\nabla u|^p dx\right)^{\frac{s}{p}} \quad \mbox{ for all $u\in X$.}
\]
Therefore, following the arguments in \cite{GAP}, if we define
\[
h(x)=\alpha_1 x^{p}-\frac{\lambda}{q}c_q^{q}x^q-\frac{1}{s}c_s^{s}x^s \quad \mbox{ for all $x\geq 0$,}
\]
then, it is
\begin{equation}\label{condizione3}
\J(u)\geq h(|\nabla u|_p).
\end{equation}

From a careful analysis of the behavior of function $h$, it follows that a constant $\lambda_1>0$ exists such that if $0<\lambda<\lambda_1$, then $h$ attains its positive maximum.

From now on, let $0<\lambda<\lambda_1$ and $R_0, R_1$ the zeros of $h$ such that $h(x)>0$ for $R_0<x<R_1$ and $h(x)\leq 0$ otherwise. Let us consider the following truncation of the functional $\J$.

Denote by $\tau:\R_{+}\to \left[0,1\right]$ a non increasing $\mathcal{C}^{\infty}$ function such that
\[
\tau(x)\equiv 1 \ \mbox{ if $x\leq R_0$,} \quad \tau(x)\equiv 0  \ \mbox{ if $x\geq R_1$.}
\]
Let $\varphi(u)=\tau(|\nabla u|_p)$ and consider the truncation functional
\[
F(u)=\int_{\Omega}A(x,u,\nabla u)\ dx-\frac{\lambda}{q}\int_{\Omega}|u|^q \ dx -\frac{1}{s}\varphi(u) \int_{\Omega}|u|^s \ dx.
\]
As in \eqref{condizione3}, we have
\begin{equation}\label{condizione4}
F(u)\geq \bar{h}(|\nabla u|_p)
\end{equation}
with 
\[\bar{h}(x)=\alpha_1 x^p-\frac{\lambda}{q}c_q^q x^q-\frac{1}{s}c_s^s x^s\tau(x) \quad \mbox{ for all $x\geq0$.}
\] 

Clearly, it is $\bar{h}\geq h$; more precisely
\begin{equation}\label{equazione3.3}
\bar{h}(x)=h(x) \ \mbox{ if $x\leq R_0$, } \ \bar{h}(x)\geq 0 \ \mbox{ if $x\geq R_0$} \ \mbox{ and } \ \bar{h}(x)=\alpha_1 x^p-\frac{\lambda}{q}c_q^q x^q \ \mbox{ if $x\geq R_1$.}
\end{equation}

The functional $F$ verifies the following properties.
\begin{lemma}\label{lemma1}
Assume that $(H_0)$--$(H_5)$ hold. Then
\begin{itemize}
\item[$(i)$] $F \in \mathcal{C}^{1}(X,\R);$
\item[$(ii)$] if $F(u)<0$, then $|\nabla u|_p<R_0$ and $\J(v)=F(v)$ for all $v$ in a small enough neighborhood of $u$;
\item[$(iii)$] $F$ satisfies $(wCPS)_c$ condition for $c<0$.
\end{itemize}
\end{lemma}
\begin{proof}
\begin{itemize}
\item[$(i)$] The proof follows from Remark \ref{remark1} and Proposition \ref{proposition1} applied to the functional $F$, since $\varphi\in \mathcal{C}^{1}(X,\R)$.
\item[$(ii)$] This property easily follows from \eqref{condizione4} and the definition of $\bar{h}$.
\item[$(iii)$] Let $(u_n)_n$ be a $(CPS)_c$ -- sequence for $F$ with $c<0$. Clearly, for $n$ large it is $F(u_n)<0$, thus from $(ii)$
\[
|\nabla u_n|_p < R_0, \J(u_n)=F(u_n) \quad \mbox{ and } \quad d\J(u_n)=dF(u_n).
\]
Hence, $(u_n)_n$ is a $(CPS)_c$ -- sequence for $\J$ which is bounded in $W_0^{1,p}(\Omega)$. As $(H_0)$--$(H_5)$ and \eqref{condizione2} hold, arguing as in the second part of \cite[Proposition 4.6]{CP2} it can be proved that a point $u\in X$ exists such that $\|u_n-u\|_p\to 0$ (up to subsequences) with $\J(u)=c$ and $d\J(u)=0$, i.e., $F(u)=c$ and $dF(u)=0$; thus $F$ satisfies the $(wCPS)_c$ condition.\qedhere
\end{itemize}
\end{proof}
\begin{proof}[Proof of the Theorem $\ref{theorem1}$] 
First of all, we note that, differently from $\J$, the functional $F$ is bounded from below (see \eqref{condizione4} and \eqref{equazione3.3}). Moreover, from $(H_5)$, taking $v\in X, v\not\equiv 0$, for all $\sigma \in \R, \sigma>0$ such that $\sigma |v|_{\infty}<\delta$, i.e., $\sigma < \frac{\delta}{|v|_{\infty}}$, it is
\[
\begin{split}
F(\sigma v)&=\int_{\Omega}A(x,\sigma v, \sigma \nabla v)\ dx - \lambda\frac{c_q^q}{q}|\sigma v|_q^q-\frac{c_s^s}{s}|\sigma v|_s^s\varphi(\sigma v)\\[10pt]
&\leq \eta_2 \sigma^p |\nabla v|^p -\lambda \frac{c_q^q}{q} \sigma^q |v|_q^q  < 0 \quad \mbox{ for $\sigma$ small enough}.
\end{split}
\]
Hence, $\displaystyle\inf_{X} F < 0$. Since $F$ satisfies the $(wCPS)_c$ condition for $c<0$, Proposition \ref{Minimum Principle} applies, then $F$ attains its minimum, i.e., $u_0\in X$ exists such that $F(u_0)=\displaystyle\min_{u\in X} F(u)<0$ and $dF(u_0)=0$.

From Lemma \ref{lemma1} $u_0$ is a local minimum also for $\J$ on $X$ with $\J(u_0)<0$. Obviously, $u_0\not\equiv 0$ as from $(H_2)$ and $(H_5)$ it is $\J(0)=F(0)=\int_{\Omega}A(x,0,0) dx=0$.

Finally, arguing as in \cite{MM}(see also \cite{MS2}), we prove that $\J$ has at least two solutions, one positive and one negative. For this, let us denote by $u_{+}=\max\{0,u\}$ and $u_{-}=\max\{-u,0\}$, the positive and the negative part of $u$, respectively, so that $u=u_{+}-u_{-}$. If we replace $\lambda |u|^{q-2}u+|u|^{s-2}u$ with $\lambda u_{+}^{q-1}+u_{+}^{s-1}$, all the previous arguments still hold true for the functional
\[
\J_{+}(u)=\int_{\Omega}A(x,u,\nabla u) dx -\frac{\lambda}{q}\int_{\Omega}u_{+}^{q}\ dx -\frac{1}{s}\int_{\Omega}u_{+}^{s}\ dx.
\]
In particular, $\J_{+}$ has a critical point $u\not\equiv 0$ (which is a local minimum for $\J_{+}$). Hence, from $(H_2)$--$(H_4)$
\[
\begin{split}
0&=\left\langle d\J_{+}(u), -u_{-}\right\rangle=\int_{\Omega}a(x,-u_{-},\nabla(-u_{-}))\cdot\nabla(-u_{-}) \ dx + \int_{\Omega}A_t(x,-u_{-},\nabla(-u_{-}))(-u_{-})\\[10pt]
&-\lambda\int_{\Omega} u_{+}^{q-1}u_{-} \ dx-\int_{\Omega} u_{+}^{s-1}u_{-} \ dx\\[10pt]
&\geq \frac{\alpha_1 \alpha_2}{\eta_1}|\nabla(-u_{-})|_p^p=\frac{\alpha_1 \alpha_2}{\eta_1}\|u_{-}\|_p^p.
\end{split}
\]
Hence, $u_{-}\equiv 0$ a.e. in $\Omega$ and so, $u$ is a positive critical point of $\J$. Similarly, replacing $\lambda |u|^{q-2}u+|u|^{s-2}u$ with $\lambda u_{-}^{q-1}+u_{-}^{s-1}$, we find a negative critical point of $\J$.
\end{proof}


\section{Existence of a mountain pass type critical point of $\J$} \label{section4}
The following compactness result holds.
\begin{proposition}\label{proposition2}
Assume that $(H_0)$--$(H_5)$ and $(H_7)$ hold. Then functional $\J$ satisfies the $(wCPS)$ condition in $\R$.
\end{proposition}
\begin{proof}
Let $(u_n)_n$ be a $(CPS)_c$ sequence for $\J$ at level $c\in\R$, i.e.,
\begin{equation}\label{condizione5}
\J(u_n)\to c \quad \mbox{ and } \quad \|d\J(u_n)\|_{X^{'}}\left(1+\|u_n\|_{X}\right)\to 0 \quad \mbox{ if $n\to +\infty$}.
\end{equation}
If $c<0$, the boundedness of $(|\nabla u_n|_p)_n$ follows as in $(iii)$ of Lemma \ref{lemma1}.

Let $c\geq 0$. Thus, condition \eqref{condizione5} implies
\begin{equation}\label{condizione6}
c+\varepsilon_n=\int_{\Omega}A(x,u_n,\nabla u_n) dx-\frac{\lambda}{q}\int_{\Omega}|u_n|^q dx-\frac{1}{s}\int_{\Omega}|u_n|^s dx
\end{equation}
and
\begin{equation}\label{condizione7}
\begin{split}
\varepsilon&=\left\langle d\J(u_n), u_n\right\rangle=\int_{\Omega}a(x,u_n,\nabla u_n) \cdot \nabla u_n dx+\int_{\Omega}A_t(x,u_n,\nabla u_n)u_n dx\\[10pt]
&-\lambda \int_{\Omega}|u_n|^q dx-\int_{\Omega}|u_n|^s dx
\end{split}
\end{equation}
where $(\varepsilon_n)_n$ denotes any infinitesimal sequence in $\R$.

From \eqref{condizione6}, \eqref{condizione7}, $(H_7)$ and $(H_2)$--$(H_4)$ we obtain
\[
\begin{split}
sc+\varepsilon_n &=s\J(u_n)-\left\langle d\J(u_n), u_n\right\rangle\\[10pt]
&=\int_{\Omega}sA(x,u_n,\nabla u_n) dx-\int_{\Omega}a(x,u_n,\nabla u_n) \cdot \nabla u_n dx-\int_{\Omega}A_t(x,u_n,\nabla u_n)u_n dx\\[10pt]
&-\lambda\left(\frac{s}{q}-1\right)\int_{\Omega}|u_n|^q dx\\[10pt]
&\geq\frac{\alpha_1\alpha_3}{\eta_1}\int_{\Omega}|\nabla u_n|^p dx -\lambda\left(\frac{s}{q}-1\right)\int_{\Omega}|u_n|^q dx
\end{split}
\]
which implies that $(u_n)_n$ is bounded in $W_0^{1,p}(\Omega)$.

The remainder of the proof follows, as in $(iii)$ of Lemma \ref{lemma1}, by arguing as in \cite[Proposition 4.6]{CP2}.
\end{proof}
Now, we state a growth estimate on $A(x,t,\xi)$ which is crucial for proving that $\J$ has a ''mountain pass'' geometry.
\begin{lemma}\label{lemma2}
Assume that $(H_0)$, $(H_2)$, $(H_3)$ and $(H_7)$ hold.

Taken $(t,\xi)\in \R\times\R^N$, a.e. in $\Omega$ and for all $\sigma\geq 1$ it is
\begin{equation}\label{condizione8}
A(x,\sigma t, \sigma \xi)\leq \sigma^{s-\frac{\alpha_3}{\eta_1}}A(x,t,\xi).
\end{equation}
\end{lemma}
\begin{proof}
Taken $(t,\xi)\in \R\times\R^N$, for any $\sigma\in\R$ and a.e. in $\Omega$, from $(H_0)$ it is
\[
\frac{d}{d\sigma}A(x,\sigma t, \sigma \xi)=A_t(x,\sigma t, \sigma \xi)t+a(x,\sigma t, \sigma \xi)\cdot\xi.
\]
Thus, by $(H_3)$ and $(H_7)$ it is
\[
\frac{d}{d\sigma}A(x,\sigma t, \sigma \xi)\leq \left(s-\frac{\alpha_3}{\eta_1}\right)\frac{1}{\sigma}A(x,\sigma t, \sigma \xi) \quad \mbox{ for all $\sigma>0$.}
\]
Hence, $(H_2)$ and simple calculation imply \eqref{condizione8}.
\end{proof}

\begin{proof}[Proof of the Theorem \ref{theorem2}] 
From Proposition \ref{proposition2} $\J$ satisfies condition $(wCPS)$ in $\R$. Moreover, Theorem \ref{theorem1} implies that for $\lambda \in ]0,\lambda_1[$ with $\lambda_1$ small enough $\J$ has a local minimum point $u_0\not\equiv 0$ such that $\J(u_0)<0$.

From \eqref{condizione3}, recalling the properties of $h$, if $\lambda \in ]0,\lambda_1[$ there exist $R, \alpha>0$ such that
\[
u\in X, \|u\|_{p}=R\Longrightarrow \J(u)\geq \alpha>0=\J(0)>\J(u_0).
\] 

Now, we prove that there exists $v\in X$ with 
\begin{equation}\label{condizione9}
\|v\|_p>\rho \quad \mbox{ and } \quad \J(v)<\J(0). 
\end{equation}
In fact, fixed $v_1\in X, v_1\not\equiv 0$, from \eqref{condizione8} it results
\[
\begin{split}
\J(\sigma v_1)&=\int_{\Omega}A(x,\sigma v_1, \sigma \nabla v_1)\ dx - \lambda\frac{\sigma^q}{q}|v_1|_q^q-\frac{\sigma^s}{s}|v_1|_s^s\\[10pt]
&\leq \sigma^{s-\frac{\alpha_3}{\eta_1}}\int_{\Omega}A(x,v_1,  \nabla v_1) dx-\frac{\sigma^s}{s}|v_1|_s^s\to -\infty \quad \mbox{ if $\sigma\to +\infty$,}
\end{split}
\] 
thus \eqref{condizione9} follows with $v=\bar{\sigma}v_1$ with $\bar{\sigma}$ large enough.

Hence, from Proposition \ref{Mountain Pass Theorem}, $\J$ has a Mountain Pass critical point $u^{*}$ such that $$\J(u^{*})\geq \alpha>0=\J(0)>\J(u_0).$$ Clearly, $u^{*}$  is a non null solution of problem \eqref{problema2}, different from $u_0$.

Finally, denoting by $\J_{+}$ the functional introduced in Section \ref{section3} and reasoning as above, it follows that $\J_{+}$ has at least two nontrivial critical points which are positive critical points of $\J$, i.e., positive solutions of problem \eqref{problema2}. Similarly, $\J_{-}$ has at least two nontrivial critical points which are negative critical points of $\J$, i.e., negative solutions of problem \eqref{problema2}.
\end{proof}


\begin{thebibliography}{99}

\bibitem{ABC} A. Ambrosetti, H. Br\'ezis and G. Cerami, Combined Effects of Concave and Convex Nonlinearities in Some Elliptic Problems, \emph{J. Funct. Anal.} \textbf{122} (1994), 519--543.

\bibitem{AB} D. Arcoya and L. Boccardo, Critical points
for multiple integrals of the calculus of variations, \emph{Arch.
Ration. Mech. Anal.} \textbf{134} (1996), 249--274.

\bibitem{BEP} L. Boccardo, M. Escobedo and I. Peral, A Dirichlet problem involving critical exponents, 
{\em Nonlinear Anal. TMA} {\bf 24} (1995), 1639--1648.


\bibitem{Br} 
H. Brezis,
\emph{Functional Analysis, Sobolev Spaces and Partial Differential Equations}, 
Universitext \textbf{XIV}, Springer, New York, 2011. 

\bibitem{CP1} A.M. Candela and G. Palmieri, {Multiple solutions
of some nonlinear variational problems}, \emph{Adv. Nonlinear Stud.} 
\textbf{6} (2006), 269--286.

\bibitem{CP2} A.M. Candela and G. Palmieri,
{Infinitely many solutions of some nonlinear variational equations},
\emph{Calc. Var. Partial Differential Equations} \textbf{34} (2009), 495--530.

\bibitem{CP3} A.M. Candela and G. Palmieri, {Some abstract critical point theorems and applications}, \emph{Discrete Contin. Dynam. Syst.} \textbf{Suppl. 2009} (2009), 133--142. 


\bibitem{CS1}
A. M. Candela and A. Salvatore, Existence o minimizer for some quasilinear elliptic problems, \emph{Discrete Contin. Dynam. Syst. Ser. S} \textbf{13}(2020), 3335-3345.


\bibitem{CS2}
A.M. Candela and A. Salvatore,
Existence of radial bounded solutions for some quasilinear elliptic
equations in $\R^N$, \emph{Nonlinear Anal.} \textbf{191} (2020), art. 111625.


\bibitem{CSS}
A.M. Candela, A. Salvatore and C. Sportelli,
Bounded solutions for quasilinear modified Schr{\"o}dinger equations, \emph{Calc. Var. Partial Differential Equations} {\bf 61} (2022), art. 220. 


\bibitem{Ca}
A. Canino,
Multiplicity of solutions for quasilinear elliptic equations, \emph{Topol. Methods Nonlinear Anal.} {\bf{6}} (1995),  357–370.

\bibitem{JJ}
M. Colin and L. Jeanjean, Solutions for a quasilinear Schrödinger equations: a dual approach,\emph{Nonlinear Anal. TMA.} {\bf{56}} (2004), 213–226.


\bibitem{deF1}
D. G. De Figueiredo, J. P. Gossez, and P. Ubilla,  Local superlinearity and sublinearity for indefinite semilinear elliptic problems, \emph{J. Funct. Anal.} \textbf{199} (2003),  452-467.


\bibitem{deF3}
D. G. De Figueiredo, J. P. Gossez, and P. Ubilla,   Multiplicity results for a family of semilinear elliptic problems under local superlinearity and sublinearity. \emph{J. Eur. Math. Soc.} \textbf{8} (2006),  269-286.


\bibitem{deF2}
D. G. De Figueiredo, J. P. Gossez, and P. Ubilla,   Local “superlinearity” and “sublinearity” for the p-Laplacian. \emph{J. Funct. Anal.} \textbf{257} (2009), 721-752.

\bibitem{GAP}
 J. Garcia Azorero and I. Peral, Multiplicity of solutions for elliptic problems with critical exponent or with a nonsymmetric term,
\emph{Trans. Amer. Math. Soc} {\textbf{323}} (1991), 877--895.

\bibitem{GAP1}
J. Garcia Azorero, I. Peral and J. J. Manfredi, Sobolev versus Hölder local minimizers and global multiplicity for some quasilinear elliptic equations,
\emph{Communic. Contemp. Math.} {\textbf{2}} (2000), 385-404.


\bibitem{MM} F. Mennuni and D. Mugnai, Leray-Lions equations of $(p,q)-$type in the entire space with unbounded potentials, \emph{Milan J. Math.} (2024) https://doi.org/10.1007/s00032-024-00391-y 



\bibitem{MS1}
 F. Mennuni and A. Salvatore, Generalized quasilinear elliptic equations in $\R^N$, \emph{Mediterr. J. Math.} \textbf{20} (2023), 205.
 
\bibitem{MS2}
F. Mennuni and A. Salvatore, Radial bounded solutions for modified Schr\"odinger equations, \emph{submitted for the publication}.

\bibitem{Per}
K. Perera, Multiplicity results for some elliptic problems with concave nonlinearities, \emph{J. Differential Equations} \textbf{140} (1997), 133-141.
\end{thebibliography}
\end{document}